\documentclass{amsart}[12pt]
\def\doctype{}

\usepackage{latexsym,amssymb,amsfonts,dsfont,bm}
\usepackage{fancyhdr,caption}
\usepackage{tikz}
\usepackage{hyperref}

\newcommand\F{\mathbb{F}}

\newcommand{\comment}[1]{}

\newcommand\sP{\mathcal{P}}
\newcommand\sM{\mathcal{M}}

\numberwithin{equation}{section}

\fancypagestyle{titlepage}{
\fancyhead[R]{\doctype}
\fancyhead[CL]{}
\cfoot{\vspace{5pt} \thepage}
}

\let\oldsection\section
\newcommand\boldsection[1]{\oldsection{\bf #1}}
\newcommand\starsection[1]{\oldsection*{\bf #1}}
\makeatletter
\renewcommand\section{\@ifstar\starsection\boldsection}
\makeatother

\newtheoremstyle{theorem}
  {12pt}		  % space above
  {0pt}  % space below
  {\sl}  % bofy font
  {\parindent}     % ident - empty=no indent,  \parindent= paragraph indent
  {\bf}  % thm head font
  {. }    % punctuation after thm head
  { }    % space after thm head: `` ``=normal \newline=linebreak
  {}     % thm head specification
\theoremstyle{theorem}
\newtheorem{thm}{Theorem}[section]  % 1st argument is your name for it
\newtheorem{lemma}[thm]{Lemma}     % 2nd argument is what is printed
\newtheorem{cor}[thm]{Corollary}

\newtheoremstyle{definition}
  {12pt}		  % space above
  {0pt}  % space below
  {}  % bofy font
  {\parindent}     % ident - empty=no indent,  \parindent= paragraph indent
  {\bf}  % thm head font
  {. }    % punctuation after thm head
  { }    % space after thm head: `` ``=normal \newline=linebreak
  {}     % thm head specification
\theoremstyle{definition}

\newtheorem{ex}[thm]{Example}

\renewcommand{\proofname}{Proof}

\makeatletter
\renewenvironment{proof}[1][\proofname]{\par
  \pushQED{\qed}%
  \normalfont \partopsep=\z@skip \topsep=\z@skip
  \trivlist
  \item[\hskip\labelsep
        \scshape
    #1\@addpunct{.}]\ignorespaces
}{%
  \popQED\endtrivlist\@endpefalse
}
\makeatother

\makeatletter
\renewcommand*\@maketitle{%
  \normalfont\normalsize
  \@adminfootnotes
  \@mkboth{\@nx\shortauthors}{\@nx\shorttitle}%
  \global\topskip42\p@\relax % 5.5pc   "   "   "     "     "
  \@settitle
  \ifx\@empty\authors \else {\vskip 1em
\vtop{\centering\shortauthors\@@par}} \fi
  \ifx\@empty\@date \else {\vskip 1em \vtop{\centering\@date\@@par}}\fi % MY CHANGE
  \ifx\@empty\@dedicatory
  \else
    \baselineskip18\p@
    \vtop{\centering{\footnotesize\itshape\@dedicatory\@@par}%
      \global\dimen@i\prevdepth}\prevdepth\dimen@i
  \fi
  \@setabstract
  \normalsize
  \if@titlepage
    \newpage
  \else
    \dimen@34\p@ \advance\dimen@-\baselineskip
    \vskip\dimen@\relax
  \fi
} % end \@maketitle
\renewcommand*\@adminfootnotes{%
  \let\@makefnmark\relax  \let\@thefnmark\relax
  \ifx\@empty\@subjclass\else \@footnotetext{\@setsubjclass}\fi
  \ifx\@empty\@keywords\else \@footnotetext{\@setkeywords}\fi
  \ifx\@empty\thankses\else \@footnotetext{%
    \def\par{\let\par\@par}\@setthanks}%
  \fi
\thispagestyle{titlepage}
}
\makeatother

\title{\large A lower bound on permutation codes\\ of distance $n-1$}

\author{Sergey Bereg}
\address{\rm Sergey Bereg:
Computer Science,
University of Texas at Dallas, Richardson, TX
}

\email{besp@utdallas.edu}
\author{Peter J.~Dukes}
\address{\rm Peter J.~ Dukes:
Mathematics and Statistics,
University of Victoria, Victoria, BC
}
\email{dukes@uvic.ca}

\thanks{Research of the first author is supported in part by NSF award CCF-1718994.
Research of the second author is supported by NSERC grant 312595--2017}

\date{\today}

\begin{document}

\begin{abstract}
A classical recursive construction for mutually orthogonal latin squares (MOLS) is shown to hold more generally for a class of permutation codes of length $n$ and minimum distance $n-1$. When such codes of length $p+1$ are included as ingredients, we obtain a general lower bound $M(n,n-1) \ge n^{1.0797}$ for large $n$, gaining a small improvement on the guarantee given from MOLS.
\end{abstract}

\maketitle
\hrule

%%%%%%%%%%%%%%%%%

%---------------------------------------------------
\section{Introduction}

Let $n$ be a positive integer.  The \emph{Hamming distance} between two permutations $\sigma, \tau \in \mathcal{S}_n$ 
is the number of non-fixed points of $\sigma \tau^{-1}$, or, equivalently, the number of disagreements when $\sigma$ and $\tau$ are written as words in single-line notation.  For example, $1234$ and $3241$ are at distance three.  

A \emph{permutation code} PC$(n,d)$ is a subset $\Gamma$ of $\mathcal{S}_n$ such that the distance between 
any two distinct elements of $\Gamma$ is at least $d$.  Language of classical coding theory is often used: elements of $\Gamma$ are \emph{words}, $n$ is the \emph{length} of the code, and the parameter $d$ is the \emph{minimum distance}, although for our purposes it is not important whether distance $d$ is ever achieved.  Permutation codes are also called \emph{permutation arrays} by some authors, where the words are written as rows of a $|\Gamma| \times n$ array.

The investigation of permutation codes essentially began with the articles \cite{DV,FD}.
After a decade or so of inactivity on the topic, permutation codes enjoyed a resurgence due to 
various applications.  See \cite{CCD,H,SM} for surveys of construction methods and for more on the coding applications.

For positive integers $n \ge d$, we let $M(n,d)$ denote the maximum size of a PC$(n,d)$.  It is easy to see that $M(n,1)=M(n,2)=n!$, and that $M(n,n)=n$.  The Johnson bound $M(n,d) \le n!/(d-1)!$ holds.  The alternating group $A_n$ shows that $M(n,3)=n!/2$.  More generally, a sharply $k$-transitive subgroup of $\mathcal{S}_n$ furnishes a permutation code $PC(n,n-k+1)$ of (maximum possible) size $n!/(n-k)!$.  For instance, the Mathieu groups $M_{11}$ and $M_{12}$ are maximum PC$(11,7)$ and PC$(12,7)$, respectively.  On the other hand, determination of $M(n,d)$ in the absence of any algebraic structure appears to be a difficult problem.  As an example, it is only presently known that $78 \le M(7,5) \le 122$; see \cite{JLOS,MBS} for details.  A table of bounds on $M(n,d)$ can be found in \cite{SM}. 

In \cite{CKL}, it was shown that the existence of $r$ mutually orthogonal latin squares (MOLS) of order $n$ yields a permutation code PC$(n,n-1)$ of size $rn$.  Although construction of MOLS is challenging in general, the problem is at least well studied.  Lower bounds on MOLS can be applied to the permutation code setting, though it seems for small $n$ not a prime power that $M(n,n-1)$ can be much larger than the MOLS guarantee.  For example, $M(6,5) =18$ despite the nonexistence of orthogonal latin squares of order six, and $M(10,9) \ge 49$, \cite{JS}, when no triple of MOLS of order 10 is known.  On the other hand, it is straightforward to see, \cite{CKL}, that $M(n,n-1)=n(n-1)$ implies existence of a full set of MOLS (equivalently a projective plane) of order $n$, so any nontrivial upper bound on permutation codes would have major impact on design theory and finite geometry.  This connection is explored in more detail in \cite{BM}.  Permutation codes are used in \cite{JS2} for some recent MOLS constructions.

Let $N(n)$ denote the maximum number of MOLS of order $n$.  Chowla, Erd\H{o}s and Strauss showed in \cite{CES} that $N(n)$ tends to infinity with $n$.  Wilson, \cite{WilsonMOLS}, found a construction strong enough to prove $N(n) \ge n^{1/17}$ for sufficiently large $n$.  Subsequently, Beth, \cite{Beth} tightened some number theory in the argument to lift the exponent to $1/14.8$.  In terms of permutation codes, then, one has $M(n,n-1) \ge n^{1+1/14.8}$ for sufficiently large $n$.

Our main result in this note gives a small improvement to the exponent.

\begin{thm}
\label{main}
$M(n,n-1) \ge n^{1.0797}$ for sufficiently large $n$.
\end{thm}

The proof is essentially constructive, although it requires, as does \cite{Beth,WilsonMOLS}, the selection of a `small' integer avoiding several arithmetic progressions.  This is guaranteed by the Buchstab sieve; see \cite{Ivt}.  Apart from this number theory, our construction method generalizes a standard design-theoretic construction for MOLS to permutation codes possessing a small amount of additional structure.  Some set up for our methodology is given in the next two sections, and the proof of Theorem~\ref{main} is given in Section~\ref{proof} as a consequence of the somewhat stronger Theorem~\ref{idem-bound}.  We conclude with a discussion of some possible next directions for this work.

%---------------------------------------------------
\section{Idempotent permutation codes and latin squares}

Let $[n]:=\{1,2,\dots,n\}$.
A \emph{fixed point} of a permutation $\pi:[n]\rightarrow [n]$ is an element $i \in [n]$ such that $\pi(i)=i$.  In single-line notation, this says symbol $i$ is in position $i$. Of course, for the identity permutation $\iota$, every element is a fixed point.

A latin square $L$ of order $n$ is \emph{idempotent} if the $(i,i)$-entry of $L$ equals $i$ for each $i \in [n]$.  Extending this definition, let us say that a permutation code is \emph{idempotent} if each of its words has exactly one fixed point.  So, a maximum PC$(n,n)$ is idempotent if and only if the `corresponding' latin square is idempotent.  

We are particularly interested in idempotent PC$(n,n-1)$ in which every symbol is a fixed point of the same number, say $r$, of words; these we call $r$-\emph{regular} and denote by $r$-IPC$(n,n-1)$.  Permutation codes with extra `distributional' properties have been investigated before.  For example, `$k$-uniform' permutation arrays are introduced in \cite{DV}, while `$r$-balanced' and `$r$-separable' permutation arrays are considered in \cite{DFKW}.  However, our definition is seemingly new, or at least not obviously related to these other conditions.

If there exists an IPC$(n,n-1)$, say $\Delta$, then $\Delta \cup \{\iota\}$ is also a PC$(n,n-1)$.  Consequently, 
$M(n,n-1) \ge rn+1$ when there exists an $r$-IPC$(n,n-1)$.  It follows that $r \le n-2$ is an upper bound on $r$.

On the other hand, if $\Gamma$ is a PC$(n,n-1)$ containing $\iota$, then the words of $\Delta$ at distance exactly $n-1$ from $\iota$ form an idempotent IPC$(n,n-1)$.  Concerning the $r$-regular condition, whether $\iota \in \Gamma$ or not, we may find an $r$-IPC$(n,n-1)$ with
\begin{equation}
\label{r-formula}
r=\max_{\sigma \in \Gamma} \min_{i \in [n]} |\{\tau \in \Gamma \setminus \{\sigma\}: \tau(i)=\sigma(i)\}|.
\end{equation}
In more detail, if $\sigma$ achieves the maximum in (\ref{r-formula}), then for each $i=1,\dots,n$ we choose exactly $r$ elements $\tau \in \Gamma$ which agree with $\sigma$ in position $i$.  After relabelling each occurrence of $\sigma(i)$ to $i$, we have the desired $r$-idempotent PC$(n,n-1)$.

A question in its own right is whether there exists an $r$-IPC$(n,n-1)$ for $r = \lfloor \frac{1}{n}(M(n,n-1)-1) \rfloor$.
However, relatively little is known about maximum permutation code sizes.  Indeed, the exact value of $M(n,n-1)$ is known only for $n=q$, a prime power, ($M(q,q-1) = q(q-1)$, \cite{FD}) and for $n=6$ ($M(6,5)=18$, \cite{K6}).

\begin{ex}
\label{ipc6}
A 2-IPC$(6,5)$:
$$\begin{array}{cc}
1\  3\  5\  6\  2\  4\ &   1\  4\  6\  2\  3\  5 \\
6\  2\  4\  5\  3\  1\ &   5\  2\  1\  3\  6\  4\\
5\  6\  3\  1\  4\  2\ &   4\  5\  3\  2\  6\  1\\
2\  5\  6\  4\  1\  3\ &   3\  6\  1\  4\  2\  5\\
3\  1\  4\  6\  5\  2\ &   6\  4\  2\  1\  5\  3\\
4\  3\  2\  5\  1\  6\  &  2\  1\  5\  3\  4\  6\\
\end{array}$$
\end{ex}

\begin{ex}
A 3-IPC$(10,9)$ (symbol `0' is used for `10'):
$$\begin{array}{ccc}
\ 1\ 8\ 6\ 2\ 9\ 5\ 4\ 0\ 3\ 7& \ 1\ 5\ 0\ 9\ 7\ 4\ 3\ 6\ 2\ 8& \ 1\ 3\ 8\ 0\ 6\ 9\ 5\ 2\ 7\ 4\\
\ 8\ 2\ 1\ 9\ 6\ 7\ 0\ 4\ 5\ 3& \ 3\ 2\ 5\ 6\ 9\ 8\ 1\ 7\ 0\ 4& \ 9\ 2\ 8\ 7\ 4\ 1\ 3\ 0\ 6\ 5\\
\ 5\ 9\ 3\ 2\ 7\ 8\ 0\ 1\ 4\ 6& \ 8\ 7\ 3\ 5\ 2\ 0\ 4\ 9\ 6\ 1& \ 9\ 4\ 3\ 1\ 6\ 2\ 8\ 5\ 0\ 7\\
\ 9\ 7\ 1\ 4\ 0\ 3\ 5\ 6\ 8\ 2& \ 6\ 3\ 2\ 4\ 1\ 0\ 9\ 7\ 5\ 8& \ 0\ 8\ 9\ 4\ 3\ 1\ 2\ 5\ 7\ 6\\
\ 0\ 3\ 6\ 7\ 5\ 2\ 1\ 4\ 8\ 9& \ 3\ 8\ 4\ 0\ 5\ 7\ 9\ 1\ 6\ 2& \ 7\ 9\ 2\ 8\ 5\ 4\ 6\ 3\ 0\ 1\\
\ 9\ 8\ 7\ 5\ 1\ 6\ 0\ 3\ 2\ 4& \ 8\ 9\ 4\ 1\ 0\ 6\ 2\ 7\ 3\ 5& \ 5\ 1\ 0\ 3\ 9\ 6\ 8\ 4\ 7\ 2\\
\ 3\ 0\ 9\ 8\ 1\ 2\ 7\ 6\ 4\ 5& \ 5\ 4\ 6\ 0\ 8\ 1\ 7\ 9\ 2\ 3& \ 8\ 6\ 0\ 2\ 4\ 3\ 7\ 5\ 1\ 9\\
\ 2\ 6\ 7\ 1\ 9\ 0\ 5\ 8\ 4\ 3& \ 4\ 1\ 9\ 0\ 2\ 3\ 6\ 8\ 5\ 7& \ 7\ 0\ 1\ 3\ 4\ 5\ 9\ 8\ 2\ 6\\
\ 0\ 7\ 4\ 6\ 1\ 5\ 8\ 2\ 9\ 3& \ 2\ 1\ 6\ 8\ 0\ 7\ 3\ 5\ 9\ 4& \ 4\ 5\ 7\ 3\ 6\ 8\ 2\ 0\ 9\ 1\\
\ 6\ 5\ 1\ 7\ 2\ 9\ 8\ 3\ 4\ 0& \ 4\ 6\ 5\ 8\ 7\ 1\ 9\ 2\ 3\ 0& \ 2\ 4\ 8\ 9\ 3\ 5\ 6\ 7\ 1\ 0\\
\end{array}$$
\end{ex}

The connection with MOLS is important in the sequel.  The following result is essentially the construction from MOLS to PC$(n,n-1)$ in \cite{CKL}, except that here we track the idempotent condition.

\begin{thm}
\label{idemp-MOLS}
If there exist $r$ mutually orthogonal idempotent latin squares of order $n$, then there exists an $r$-IPC$(n,n-1)$.
\end{thm}

\begin{proof}
Suppose $L_1,\dots,L_r$ are the hypothesized latin squares, each on the set of symbols $[n]$.  For each $i \in [n]$ and $j \in [r]$, define the permutation $\pi_{i,j} \in \mathcal{S}_n$ by $\pi_{i,j}(x) = y$ if and only if the $(x,y)$-entry of $L_j$ is $i$.  Let $\Gamma = \{\pi_{i,j} : i \in [n], j \in [r]\}$. Consider distinct permutations $\pi_{i,j}$ and $\pi_{h,k}$ in $\Gamma$.  They have no agreements if $j=k$, by the latin property, and they have exactly one agreement if $j \neq k$ by the orthogonality of squares $L_j$ and $L_k$.  So $\Gamma$ is a PC$(n,n-1)$.  Moreover, since each $L_j$ is an idempotent latin square, the permutation $\pi_{i,j}$ has only the fixed point $i$.  It follows that $\Gamma$ is in fact an $r$-IPC$(n,n-1)$.
\end{proof}

We remark that the maximum number of mutually orthogonal idempotent latin squares of order $n$ is either $N(n)$ or $N(n)-1$, since we may permute rows and columns of one square so that its main diagonal is a constant, and then permute symbols of the other squares.  That is, our idempotent condition is negligible as far as the rate of growth of $r$ in terms of $n$ is concerned.

\begin{cor}
\label{prime-powers}
For prime powers $q$, there exists a $(q-2)$-IPC$(q,q-1)$.
\end{cor}

MacNeish's bound for MOLS is an application of the standard product construction for MOLS with prime-power ingredients.

\begin{thm}[MacNeish's bound; see \cite{CES,CD,WilsonMOLS}]
\label{macneish}
If $n=q_1\dots q_t$ is factored as a product of powers of distinct
primes, then $N(n) \ge q-1$, where $q= \min\{q_i: i =1,\dots,t\}$.
\end{thm}

From Corollary~\ref{prime-powers}, we immediately have a similar result for idempotent permutation codes.

\begin{thm} 
\label{product}
If $n=q_1\dots q_t$ is factored as a product of powers of distinct
primes, then there exists a $(q-2)$-IPC$(n,n-1)$ where  
$q=\min\{q_i : i=1,\dots,t\}$.
\end{thm}

Finally, it is worth briefly considering a `reverse' of the MOLS construction for PC$(n,n-1)$.  Suppose a PC$(n,n-1)$, say $\Gamma$, is partitioned into PC$(n,n)$, say $\Gamma_1,\dots,\Gamma_r$. We define $r$ partial latin squares as linear combinations of permutation matrices for $\Gamma_i$ with symbolic coefficients.  Since two distinct words of the code have at most one agreement, overlaying any two of the $r$ partial latin squares leads to distinct ordered pairs of symbols over the common non-blank cells. 
We merely offer an example, but remark that this viewpoint is helpful for our recursive construction to follow.

\begin{ex}
The $2$-IPC$(6,5)$ of Example~\ref{ipc6} admits a partition into three disjoint PC$(6,6)$; this can be seen by reading the array four rows at a time.  Each of these 
sub-arrays is converted into a partial latin square of order six, where a permutation $\pi$ having fixed point $i$ fills all cells of the form $(x,\pi(x))$ in its square with symbol $i$.  
$$
\begin{array}{|c|c|c|c|c|c|}
\hline
1&4&&&3&2\\
\hline
&2&1&&4&3\\
\hline
&&3&2&1&4\\
\hline
3&&&4&2&1\\
\hline
4&1&2&3&&\\
\hline
2&3&4&1&&\\
\hline
\end{array}
\hspace{1cm}
\begin{array}{|c|c|c|c|c|c|}
\hline
1&&5&6&2&\\
\hline
5&2&6&1&&\\
\hline
2&6&&5&&1\\
\hline
&1&2&&6&5\\
\hline
6&&1&&5&2\\
\hline
&5&&2&1&6\\
\hline
\end{array}
\hspace{1cm}
\begin{array}{|c|c|c|c|c|c|}
\hline
&6&4&3&&5\\
\hline
6&&&5&3&4\\
\hline
4&5&3&&6&\\
\hline
5&3&6&4&&\\
\hline
&4&&6&5&3\\
\hline
3&&5&&4&6\\
\hline
\end{array}
$$
\end{ex}

%---------------------------------------------------
\section{A recursive construction using block designs}

In this section, we observe that idempotent permutation codes can be combined to produce larger such codes.  Since the resultant code must preserve at most one agreement between different words, we are naturally led to consider block designs to align the ingredient codes.

A \emph{pairwise balanced design} PBD$(n,K)$ is a pair $(V,\mathcal{B})$, where $V$ is a set of size $n$, $\mathcal{B}$ is a family of subsets of $V$ with sizes in $K$, and such that every pair of distinct elements of $V$ belongs to exactly one set in $\mathcal{B}$.  The sets in $\mathcal{B}$ are called \emph{blocks}.  Thinking of a PBD as a special type of hypergraph, we refer the elements of $V$ as \emph{vertices} or \emph{points}.

The following construction is inspired from a similar one for MOLS; see \cite[Theorem 3.1]{CD}.

\begin{thm}
\label{pbd-construction}
If there exists a PBD$(n,K)$ and, for every $k \in K$, there exists an $r$-IPC$(k,k-1)$, then there exists an $r$-IPC$(n,n-1)$.
\end{thm}

\begin{proof} 
Let $([n],\mathcal{B})$ be a PBD$(n,K)$.  For each block $B \in \mathcal{B}$, take a copy of an $r$-IPC$(|B|,|B|-1)$ on the symbols of $B$.  Its permutations are, say, $\pi_{i,j}^B:B \rightarrow B$, for $i \in B$, $j =1,\dots,r$, where $\pi_{i,j}^B(i)=i$ is the unique fixed point for $\pi_{i,j}^B$.

Let $i \in [n]$ and put $\mathcal{B}_i:=\{B \setminus \{i\}: B \in \mathcal{B}, i \in B\}$.  Since $\mathcal{B}$ is the block set of a PBD, we have that $\mathcal{B}_i$ is a partition of $[n] \setminus \{i\}$.   For $j=1,\dots,r$, define a permutation $\pi_{i,j}:[n] \rightarrow [n]$ by
$$\pi_{i,j}(x) = \begin{cases}
i & \text{if~} x =i, \\
\pi^B_{i,j}(x) & \text{if~} x\neq i, \text{~where~} x \in B \in \mathcal{B}_i.
\end{cases}
$$
We claim that $\{\pi_{i,j}:i \in [n], j \in [r]\}$ is an $r$-IPC$(n,n-1)$ such that, for each $i$, the subset  $\{\pi_{i,j}: j \in [r]\}$ has precisely the fixed point $i$.
First, each $\pi_{i,j}$ is a permutation.  That $([n],\mathcal{B})$ is a PBD ensures that $\pi_{i,j}$ is well-defined and bijective.  In particular, if $a \in [n]$, $a \neq i$, we have $\{i,a\}$ contained in a unique block, say $A \in \mathcal{B}$.  

It remains to check the minimum distance.
Consider $\pi_{i,j}$ and $\pi_{i,j'}$ for $j \neq j'$.
They agree on $i$, but suppose for contradiction that they agree also on $h \neq i$.  Let $B$ be the unique block of $\mathcal{B}_i$ containing $h$.  By construction, we must have $\pi^B_{i,j}$ agreeing with $\pi^B_{i,j'}$ at $h$, and this is a contradiction to the minimum distance being $|B|-1$ within this component code.
 
Now, consider $\pi_{i,j}$ and $\pi_{i',j'}$ for $i \neq i'$.  Suppose they agree at distinct positions $h$ and $l$.  Say $\pi_{i,j}(h)=\pi_{i',j'}(h)=a$ and $\pi_{i,j}(l) = \pi_{i',j'}(l) = b$.  Then $\{i,i',h,a\}$ and $\{i,i',l,b\}$ are in the same block.  It follows that $h,l$ are in the same block and we get a contradiction again.
\end{proof}

We illustrate the construction of Theorem~\ref{pbd-construction}.

\begin{ex}
Figure~\ref{idpc10} shows a PBD$(10,\{3,4\})$ at left.
The design is built from an affine plane of order three (on vertex set $\{1,\dots,9\}$) with one parallel class extended (to vertex $0$).  
In the center, template idempotent permutation codes of lengths 3 and 4 are shown.  The code of length three is simply an idempotent latin square, but note that the code of length four achieves minimum distance three.
On the right is shown the resultant $1$-IPC$(10,9)$, an unimpressive code for illustration only.  It can be checked that two rows agree in at most one position (which if it exists is found within the unique block containing the chosen row labels).
\end{ex}

\begin{figure}[htbp]
\begin{minipage}{.4\linewidth}
\begin{center}
$$
\begin{array}{cccc}
\{0,1,2,3\} & \{1,4,7\} & \{1,5,9\} & \{1,6,8\} \\
\{0,4,5,6\} & \{2,5,8\} & \{2,6,7\}  & \{2,4,9\}  \\
\{0,7,8,9\} & \{3,6,9\}  & \{3,4,8\}  & \{3,5,7\}  \\
\end{array}
$$

\vspace{4mm}

\begin{tikzpicture}
\draw (-2.5,0)--(1,0);
\draw (-2.5,0) to [out=45,in=180] (-1,1);
\draw (-2.5,0) to [out=-45,in=180] (-1,-1);

\draw (-1,1)--(1,1);
\draw (-1,-1)--(1,-1);
\draw (0,-1)--(0,1);
\draw (1,-1)--(1,1);
\draw (-1,-1)--(-1,1);
\draw (1,1)--(-1,-1);
\draw (1,-1)--(-1,1);

\draw (0,1)--(-1,0);
\draw (-1,0) to [out=225,in=90] (-1.5,-1);
\draw (1,-1) to [out=225,in=270] (-1.5,-1);

\draw (0,-1)--(1,0);
\draw (1,0) to [out=45,in=270] (1.5,1);
\draw (-1,1) to [out=45,in=90] (1.5,1);

\draw (0,1)--(1,0);
\draw (1,0) to [out=315,in=90] (1.5,-1);
\draw (-1,-1) to [out=315,in=270] (1.5,-1);

\draw (0,-1)--(-1,0);
\draw (-1,0) to [out=135,in=270] (-1.5,1);
\draw (1,1) to [out=135,in=90] (-1.5,1);

\node at (-2.5,0.3) {$0$};
\node at (-0.3,0) {$5$};
\node at (0,1.3) {$2$};
\node at (0,-1.3) {$8$};
\node at (.7,0) {$6$};
\node at (1,1.3) {$3$};
\node at (1,-1.3) {$9$};
\node at (-1.3,0) {$4$};
\node at (-1,1.3) {$1$};
\node at (-1,-1.3) {$7$};

\filldraw (-2.5,0) circle [radius=.1];
\filldraw (0,0) circle [radius=.1];
\filldraw (0,1) circle [radius=.1];
\filldraw (0,-1) circle [radius=.1];
\filldraw (1,0) circle [radius=.1];
\filldraw (1,1) circle [radius=.1];
\filldraw (1,-1) circle [radius=.1];
\filldraw (-1,0) circle [radius=.1];
\filldraw (-1,1) circle [radius=.1];
\filldraw (-1,-1) circle [radius=.1];
\end{tikzpicture}
\caption*{PBD$(10,\{3,4\})$}
\end{center}
\end{minipage}
~
\begin{minipage}{.25\linewidth}
\begin{center}
\begin{tabular}{c}
a\ c\ b\\
c\ b\ a\\
b\ a\ c\\
\end{tabular}

\vspace{1cm}

\begin{tabular}{c}
a\ c\ d\ b\\
c\ b\ d\ a\\
d\ a\ c\ b\\
b\ c\ a\ d\\
\end{tabular}
\caption*{ingredient codes} 
\end{center}
\end{minipage}
~
\begin{minipage}{.3\linewidth}
\begin{center}
$$
\begin{array}{c}
0\ 2\ 3\ 1\ 5\ 6\ 4\ 8\ 9\ 7\\
2\ 1\ 3\ 0\ 7\ 9\ 8\ 4\ 6\ 5\\
3\ 0\ 2\ 1\ 9\ 8\ 7\ 6\ 5\ 4\\
1\ 2\ 0\ 3\ 8\ 7\ 9\ 5\ 4\ 6\\
5\ 7\ 9\ 8\ 4\ 6\ 0\ 1\ 3\ 2\\
6\ 9\ 8\ 7\ 0\ 5\ 4\ 3\ 2\ 1\\
4\ 8\ 7\ 9\ 5\ 0\ 6\ 2\ 1\ 3\\
8\ 4\ 6\ 5\ 1\ 3\ 2\ 7\ 9\ 0\\
9\ 6\ 5\ 4\ 3\ 2\ 1\ 0\ 8\ 7\\
7\ 5\ 4\ 6\ 2\ 1\ 3\ 8\ 0\ 9\\
\end{array}
$$
\caption*{resultant code} 
\end{center}
\end{minipage}
\caption{Recursive construction of a 1-IPC$(10,9)$.}
\label{idpc10}
\end{figure}

We conclude this section with an existence result for pairwise balanced designs, to which we can apply Theorem~\ref{pbd-construction}.  This is implicit in early constructions of mutually orthogonal latin squares, \cite{CES,WilsonMOLS}, but we provide a proof for completeness.

\begin{lemma}
\label{pbd-existence}
Suppose $m,t,u$ are integers satisfying $N(t) \ge m-1$ and $0 \le u \le t$.  Then there exists a PBD$(mt+u,\{m,m+1,t,u\})$.
\end{lemma}

\begin{proof}
Let us take MOLS $L_1,\dots,L_{m-1}$ of order $t$. On the set of points $[t] \times [m+1]$, define the family of sets
$$B_{ij}=\{(i,1),(j,2),(L_1(i,j),3),\dots,(L_{m-1}(i,j),m-1)\},~i,j \in [t].$$
By a \emph{fiber} we mean a set of the form $F_h=\{(x,h): x\in [t]\}$.  Any two points from different fibers occur together in exactly one such block, by the properties of MOLS.  Now, delete all but $u$ points from the last fiber, so that our point set is now $V=[t] \times [m] \cup [u] \times \{m+1\}$.  Let $F'_h=F_h$ for $h =1,\dots,m$ and put $F'_{m+1}=\{(x,m+1): x\in [u]\}$.
For each $B_{ij}$, truncate a deleted point (if present) to produce $B'_{ij}$.  We claim that the $B'_{ij}$, $i,j \in [t]$, together with $F'_h$, $h \in [m+1]$, form the blocks of a PBD on $V$.  Consider a pair of distinct elements in $V$.  If they are in different fibers, they belong to exactly one block of the form $B'_{ij}$ (of size $m$ or $m+1$), and if they are in the same fiber $F'_h$, they are in this block (of size $t$ or $u$).
\end{proof}

%---------------------------------------------------
\section{An improved exponent}
\label{proof}

We apply the partition and extension technique from \cite{BMS} to construct an idempotent permutation code. We briefly summarize the method as needed for our use to follow.
Let $\Gamma$ be a PC$(n,n-1)$, say on symbol set $[n]$.  Consider a partition $\mathcal{P}=\{P_1,\dots,P_k\}$ of $[n]$ and a family of disjoint subsets $\mathcal{M}=\{M_1,\dots,M_k\}$ of $\Gamma$ such that
\begin{itemize}
\item
for each $i=1,\dots,k$, the Hamming distance between distinct elements of $M_i$ is $n$;
\item
for each $i=1,\dots,k$ and every $\sigma \in M_i$, there exists $z \in P_i$ such that $\sigma(z) \in P_i$.
\end{itemize}
For $\sigma \in M_i$, we define its \emph{extension}, $\mathrm{ext}(\sigma)$, a permutation $\sigma'$ on $[n] \cup \{\infty\}$, by 
$$\sigma'(x):=\begin{cases}
\sigma(x)&\text{if~}x \neq z,\infty,\\
\infty &\text{if~}x = z,\\
\sigma(z)&\text{if~}x =\infty,
\end{cases}$$
where, $z$ is some element in $P_i$ such that $\sigma(z)\in P_i$.  Observe that $d(\sigma',\tau') \ge n$ for any $\sigma,\tau \in \cup_{i=1}^k M_i$.
With $\Pi = (\mathcal{P},\mathcal{M})$, we define 
$$\mathrm{ext}(\Pi):=\{\mathrm{ext}(\sigma): \sigma \in \cup_{i=1}^k M_i\},$$
a PC$(n+1,n)$.  Of course, we may for convenience use permutations on other sets than $[n]$.  
One natural choice is to use a finite field with the affine linear group of permutations.  
The bound $M(q^2+1,q^2)\ge q^3+q^2$ for a prime power $q$ was obtained earlier using the above method, \cite{B17}.  
We construct a similar (actually slightly weaker) idempotent code, with the proof provided for completeness.

\begin{thm} 
\label{q2}
\label{baer}
For any prime power $q$, there exists a $(q-1)$-IPC$(q^2+1,q^2)$.
\end{thm}

\begin{proof}
Let $\F_{q^2}$ denote the field of order $q^2$.  
Consider $\Pi = (\mathcal{P},\mathcal{M})$, where 
$\sP$ is taken to be the partition of $\F_{q^2}$ into additive cosets of the subfield $\F_q$, and where
$\sM$ is the union of any $q$ subsets of AGL$(1,q^2)$ of the form $M_a=\{x \mapsto ax+b:b \in \F_{q^2}\}$, where $a$ runs over $q$ elements in $\F_{q^2} \setminus \F_q$.  We remark that $\sM$ is a PC$(q^2,q^2-1)$ partitioned into $q$ cosets of the cyclic subgroup $M_1$, a PC$(q^2,q^2)$.

For such elements $a$, we claim that $|\{(x,ax+b):x, ax+b \in c+\F_q\} | = 1$ for any $b,c \in \F_{q^2}$.  First, suppose $x,y \in c+\F_q$, $x \neq y$,
with $ax+b,ay+b \in c+\F_q$.  Then $x-y,a(x-y) \in \F_q$, and so $a \in \F_q$, a contradiction.  So these sets have size at most one.  To see that they are nonempty, fix $c$ and note that there are $q$ choices for $x \in c+\F_q$ and, for each, $q$ choices for $b$ so that $ax+b \in c+\F_q$.

The extension ext$(\Pi)$ is a PC$(q^2+1,q^2)$ by \cite[Theorem 1]{BMS}.  
A $(q-1)$-IPC$(q^2+1,q^2)$ can be obtained from it as follows.
Every permutation in $M_a$ has a unique fixed point since $a\ne 1$.
A permutation in $M_a$ may lose a fixed point after extension if that position gets replaced by $\infty$.  
We remove such new permutations without fixed points. 
Since every set $P\in \sP$ has size $q$, at most $q$ permutations are removed in this way from each ext$(M_a)$. 
Since sets in $\sP$ are disjoint, every symbol of $\F_{q^2}$
is a fixed point of at least $q-1$ of the remaining permutations.  
By removing some permutations if necessary, we can ensure every symbol of $\F_{q^2}$
is a fixed point exactly $q-1$ times. 

Finally, we choose the last set for $\sM$. Pick any $q-1$ permutations from coset $M_1$ other than 
the identity. Adjoin the new symbol $\infty$ at the end of each of these permutations. 
Then $\infty$ will be the only fixed point and the entire permutation code is a $(q-1)$-IPC$(q^2+1,q^2)$. 
\end{proof}

We remark that it is also possible to get an $r$-IPC$(q+1,q)$ for \emph{primes} $q$, where $r = O(\sqrt{q})$.
For this, take the construction from \cite[Section 4]{BMS} and change it as in the proof of Theorem \ref{q2}
by removing permutations without fixed points and using $O(\sqrt{q})$ permutations from coset 
$M_1=\{x \mapsto x+b:b \in \F_{q}\}$ for the last set for $\sM$. 

Next we cite an important number-theoretic result used in \cite{Beth} for MOLS.

\begin{lemma}[Buchstab sieve; see \cite{Ivt}]
\label{buchstab}
Let $2=p_0,p_1,\dots,p_k$ be the primes less than or equal to $y$, and let $\omega=\{a_0,a_1,\dots,a_k,b_1,\dots,b_k\}$ be a set of $2k+1$ integers.
Let $B_\omega (x,y)$ denote the number of positive integers $z \le x$ which do not lie in any of the arithmetic progressions $z  \equiv a_i \pmod{p_i}$, $i=0,1,\dots,k$ or $z \equiv b_j \pmod{p_j}$, $j=1,\dots,k$.
Then $B_\omega(x,x^{4.2665})$ tends to infinity with $x$, independent of the selections $\omega$.
\end{lemma}

The tools are now in place for our asymptotic lower bound on $M(n,n-1)$.

\begin{thm}
\label{idem-bound}
For sufficiently large $n$, there exists an $r$-IPC$(n,n-1)$ with $r \ge n^{0.0797}$.
\end{thm}

\begin{proof}
We follow a similar strategy as in \cite{Beth,WilsonMOLS}, applying the Buchstab sieve.  The main idea is to write $n$ in the form $mt+u$, where these parameters satisfy the conditions of Lemma~\ref{pbd-existence} while being each a product of large enough prime powers.

To simplify the calculations below, put $\beta := 4.2665$ and $\gamma = 0.0797$, and note that $2\gamma(\beta+2)<1$.

Let $r=\lceil n^\gamma \rceil$.  Choose a prime power $q$ satisfying $(r+1) \le q \le 2(r+1)$ and $q \not\equiv n \pmod{2}$.  Indeed, Bertrand's postulate allows for $q$ prime when $n$ is even, and otherwise the interval permits a choice in which $q$ is a power of $2$.
Let $m:=q^2$.  In view of Corollary~\ref{prime-powers} and Theorem~\ref{baer}, there exist both $r$-IPC$(m,m-1)$ and $r$-IPC$(m+1,m)$.  It is important for an estimate to follow that $m=O(n^{2\gamma})$, and so in particular $m+1<n^{1/(\beta+2)}$ for sufficiently large $n$.

Now, use Lemma~\ref{buchstab} with $m$ taking the role of $y$, $a_i = -\lfloor \frac{n}{m+1} \rfloor$, and $b_j = m_jn-\lfloor \frac{n}{m+1}  \rfloor$, where $m_j$ denotes the multiplicative inverse of $m$ modulo $p_j$.  We remark that, since $m$ and $n$ have opposite parity, these choices are consistent for $p=2$; that is, $a_0 \equiv b_0 \pmod{2}$.  
The conclusion of Lemma~\ref{buchstab} gives that, for sufficiently large $n$, there exists a positive integer $t' \le m^\beta$, so that, with $t:=t'+\lfloor \frac{n}{m+1} \rfloor$, we have $t \not\equiv 0 \pmod{p}$ and $mt \not\equiv n \pmod{p}$ for each prime $p \le m$.  The first congruence conditions is immediately equivalent to $t' \not\equiv a_i \pmod{p_i}$ and the second comes from $mt = m(t'+\lfloor \frac{n}{m+1} \rfloor) \equiv n$ iff $t' \equiv m_jn-\lfloor \frac{n}{m+1} \rfloor = b_j \pmod{p_j}$.

Put $u=n-mt$ so that $n=mt+u$.  Since $t \ge \frac{n}{m+1}$, we have $u \le t$.  And $u$ is nonnegative since
\begin{align*}
u = n-mt &\ge n-m \left( m^\beta + \frac{n}{m+1} \right) \\
 &=\frac{n}{m+1} - m^{\beta+1}\\
 & > n^{1-1/(\beta+2)} - n^{(\beta+1)/(\beta+2)}= 0.
\end{align*}
Recall that $t$ is divisible by no primes less than or equal to $m$.  It follows by Theorem~\ref{product} that there exists an $(m-1)$-IPC$(t,t-1)$ and hence, since $m-1 \ge r$, an $r$-IPC$(t,t-1)$. Indeed, we have $N(t) >m$ from MacNeish's bound, Theorem~\ref{macneish}.  We also chose $t$ so that $u=n-mt$ is divisible by no primes less than or equal to $m$.  From this, we likewise obtain an $r$-IPC$(u,u-1)$.  By Lemma~\ref{pbd-existence}, there exists a PBD$(mt+u,\{m,m+1,t,u\})$.  Hence, by Theorem~\ref{pbd-construction}, there exists an $r$-IPC$(n,n-1)$.
\end{proof}

Our main result, Theorem~\ref{main}, stating that $M(n,n-1) \ge n^{1.0797}$ for large $n$, is now an immediate consequence of Theorem~\ref{idem-bound}.

\section{Discussion}

Our exponent $0.0797$ is only slightly better than $1/14.8 \approx 0.0675$ already known for MOLS.  However, in certain cases it may be possible to construct a PBD whose block sizes are large primes or primes plus one.  For example, a projective plane of order $p$ is a PBD$(p^2+p+1,\{p+1\})$.  If $p'$ is another prime, say with $\sqrt{2p} < p' < p$, then, by deleting all but $p'$ points from one line of this plane we obtain a PBD$(p^2+p',\{p',p,p+1\})$.  Our construction gives an $r$-IPC$(n,n-1)$ with $r$ on the order of $n^{1/4}$, and this is not in general subsumed by existing MOLS bounds nor existing permutation code constructions.  A little more generally, an exponent approaching 1/4 can be achieved when $n$ has a representation $n=p_1+p_2 p_3$ for primes $p_i$ satisfying $n^{1/2-\epsilon} < p_1 < \max\{p_2,p_3\} <n^{1/2+\epsilon}$.

The exponent could also be improved if a better construction for designs with large block sizes could be used in place of Lemma~\ref{pbd-existence}.  Even with our family of designs from Lemma~\ref{pbd-existence}, the hypothesis $N(t) \ge m-1$ significantly harms our exponent.  
Wilson's construction for MOLS in \cite{WilsonMOLS} drops this strong requirement on $t$.  However, a preliminary look at the construction suggests that a suitable relaxation for permutation codes PC$(n,n-1)$ is likely to demand a partition into codes of full distance, so that some latin square structure is maintained.  This is an idea worth exploring in future work.

In another effort to work around the hypothesis $N(t) \ge m-1$, we explored the idea of letting $t=s^2$ for an integer $s$ with no prime factors up to about $\sqrt{m}$.  Our remainder $u=n-ms^2$ is then a quadratic in $s$ and one must avoid an extra arithmetic progression.  The allowed range for $s$ is too small for the trade-off to be worthwhile.

Applying equation (\ref{r-formula}) to a known permutation code with $n=60$, we can report the existence of a $6$-IPC$(60,59)$.  By comparison, it is only known that $N(60) \ge 5$; see \cite{Abel}.  As a next step in researching $r$-IPC$(n,n-1)$, it would be interesting to accumulate some additional good examples, primarily in the case when neither $n$ nor $n-1$ is a prime power.

Finding a maximum idempotent code (with the assumption on $r$-regularity dropped) is closely related to finding a smallest maximal set of permutations at distance $n$ in a PC$(n,n-1)$.  Some preliminary experiments on known codes suggest that it is sometimes possible to have one permutation at distance exactly $n-1$ to all others.   As one example, the current lower bound on $M(54,53)$ is $408$ (see \cite{BMS}), yet there is an idempotent code of size $407$.

Finally, we remark that using designs to join permutation codes may be a fruitful approach not only for smaller Hamming distances, but also perhaps for other measures of discrepancy, such as the Lee metric.

\section*{Acknowledgements}

We thank the referees for careful reading and helpful suggestions which improved the presentation.

%%%%%%%%%%


\begin{thebibliography}{99}

\bibitem{Abel}
R.J.R. Abel, Existence of five MOLS of orders 18 and 60.
\emph{J. Combin. Des.} 23 (2015), 135--139.

\bibitem{B17}
S.~Bereg, 
Extending permutation arrays for even prime powers. \emph{Manuscript} 2017. 

\bibitem{BMS}
S.~Bereg, L.~Morales and I.H.~Sudborough, 
Extending permutation arrays: improving MOLS bounds. \emph{Des. Codes Cryptogr.} 83 (2017), 661--683. 

\bibitem{Beth}
T.~Beth,
Eine Bemerkung zur Abschätzung der Anzahl orthogonaler lateinischer Quadrate mittels Siebverfahren.
\emph{Abh. Math. Sem. Univ. Hamburg} 53 (1983), 284--288.

\bibitem{BM}
J.~Bierbrauer and K.~Metsch, 
A bound on permutation codes.
\emph{Electron. J. Combin.} 20 (2013), P6, 12 pp. 

\bibitem{CES}
S.~Chowla, P.~Erd\H os, and E.G.~Strauss, On the maximal number of pairwise orthogonal latin squres of a given order. 
{\em Canad. J. Math.} 12 (1960), 204--208.

\bibitem{CCD}
W.~Chu, C.J.~Colbourn, and P.J.~Dukes,
Permutation codes for powerline communication.
{\em Des. Codes Cryptography} 32 (2004), 51--64.

\bibitem{CD}
C.J.~Colbourn and J.H.~Dinitz, Making the MOLS table.  \emph{Computational and constructive design theory},  67--134, Math. Appl., 368, Kluwer Acad. Publ., Dordrecht, 1996.

\bibitem{CKL}
C.J.~Colbourn, T.~Kl\o ve, and A.C.H.~Ling, Permutation
arrays for powerline communication and mutually orthogonal Latin squares.
{\em IEEE Trans. Inform. Theory} 50 (2004), 1289--1291.

\bibitem{DV}
M.~Deza and S.A.~Vanstone, Bounds for permutation arrays. {\em J.
Statist. Plann. Inference} 2 (1978), 197--209.

\bibitem{DFKW}
C.~Ding, F.-W.~Fu, T.~Kl\o ve and V.K.-W.Wei, Constructions of permutation arrays. 
\emph{IEEE Trans. Inform. Theory} 48 (2002), 977--980.

\bibitem{FD}
P.~Frankl and M.~Deza,
On the maximum number of permutations with given maximal or minimal 
distance. {\em J. Combin. Theory Ser.
A} 22 (1977), 352--360.

\bibitem{H}
S.~Huczynska, Powerline communication and the 36 officers problem. \emph{Philos. Trans. R. Soc. Lond. Ser. A Math. Phys. Eng. Sci.} 364 (2006), 3199--3214.

\bibitem{Ivt}
H.~Iwaniec, J. van de Lune and H.J.J.~ te Riele, The limits of Buchstab's iteration sieve. 
\emph{Nederl. Akad. Wetensch. Indag. Math.} 42 (1980), 409--417.

\bibitem{JLOS}
I. Janiszczak, W. Lempken, P.R.J. \"{O}sterg\aa rd and R. Staszewski, 
Permutation codes invariant under isometries, \emph{Des. Codes Cryptogr.} 75 (2015), 497--507. 

\bibitem{JS}
I.~Janiszczak and R.~Staszewski, An improved bound for permutation arrays of length 10. Preprint 4, 
Institute for Experimental Mathematics, University Duisburg-Essen, 2008.

\bibitem{JS2}
I.~Janiszczak and R.~Staszewski,
Isometry invariant permutation codes and mutually orthogonal Latin squares,
\url{https://arxiv.org/abs/1812.06886}.

\bibitem{K6}
T.~Kl\o ve, Classification of permutation codes of length 6 and minimum
distance 5. {\em Proc. Int. Symp. Information Theory Appl.}, 2000,
465--468.

\bibitem{MBS}
R.~Montemanni, J.~Barta, and D.H.~Smith, Permutation codes: a new
upper bound for M(7,5), Proc. of the 2014 Internat. Conf.
on Informatics and Advanced Computing (ICIAC), T. Yingthawornsuk
and O. Adiguzel eds., Internat. Academy of Engineers (IA-E), pages
1-3, 2014. 

\bibitem{SM}
D.H.~Smith and R.~Montemanni,
A new table of permutation codes. 
\emph{Des. Codes Cryptogr.} 63 (2012), 241--253.

\bibitem{WilsonMOLS}
R.M. Wilson, Concerning the number of mutually orthogonal Latin squares.
\emph{Discrete Math.}
9 (1974), 181-198.

\end{thebibliography}
\end{document}